\documentclass[letterpaper, 10 pt, conference]{ieeeconf}  % Comment this line out
\IEEEoverridecommandlockouts                              % This command is only
\title{\LARGE \bf
Consensus seeking gradient descent flows on boundaries of convex sets}
\author{Johan Markdahl% <-this % stops a space
\thanks{This work is supported by the UL project OptBioSys.}% <-this % stops a space
\thanks{J. Markdahl is with the Luxembourg Centre for Systems Biomedicine at the University of Luxembourg. Email: 
    {\tt\small markdahl@kth.se}}%
}

\usepackage{cite}

\usepackage{subdepth} % fixes issues with superscript and subscript vertical position when both are used (there is still some issue with horizontal placement, look at $\|x\|_2^2$)

\usepackage{graphics} % for pdf, bitmapped graphics files
\usepackage{epsfig} % for postscript graphics files

\usepackage{amssymb} % amsfonts % assumes amsmath package installed
\usepackage{amsmath,psfrag,color,graphicx} %euscript

\usepackage{thmtools}
\renewcommand\thmcontinues[1]{Continued}

\usepackage{graphicx}

\usepackage[small,hang,bf,up,labelsep=period]{caption} % have not checked that it works, no figures yet
\usepackage{quoting} % for abstracts
\usepackage{lettrine} % for the first letter
\usepackage{tocloft} % for the table of contents title
\usepackage{titletoc}
\usepackage{fmtcount}
\usepackage{calc} % to use arithmics with counters
\usepackage{cite} 
\usepackage{multicol} % this could be used to create two column chapters
\usepackage{centernot}
\usepackage{nicefrac}
\usepackage{dsfont}
\usepackage{booktabs}
\usepackage{mathtools, cuted}

\usepackage{lipsum}
\usepackage{psfrag}
\usepackage[export]{adjustbox} %\adjincludegraphics is more powerful than includegraphics
\PassOptionsToPackage{cmyk}{xcolor}
\usepackage[cmyk]{xcolor}
\selectcolormodel{cmyk}

\usepackage{textcomp}

\usepackage{txfonts}
\usepackage{lmodern}
\usepackage{times}

\usepackage{scalerel}[2016/12/29]
\DeclareFontEncoding{LS1}{}{}
\DeclareFontSubstitution{LS1}{stix}{m}{n}

\makeatletter
\newcommand*\textmathversion{\csname textmv@\math@version\endcsname}
\newcommand*\textmv@normal{m}
\newcommand*\textmv@bold{b}
\makeatother

\newcommand{\ve}[2][]{\ensuremath{\boldsymbol{\mathrm{#2}}}_{#1}}
\newcommand{\vet}[2][]{\ensuremath{\smash{\boldsymbol{\mathrm{#2}}^{\!\top}_{#1}}}}
\newcommand{\vd}[2][]{\ensuremath{\dot{\boldsymbol{\mathrm{#2}}}_{#1}}}

\newcommand{\ma}[2][]{\ensuremath{\boldsymbol{\mathrm{#2}}}_{#1}}
\newcommand{\mat}[2][]{\ensuremath{\boldsymbol{\mathrm{#2}}^{\!\top}_{#1}}}

\newcommand{\R}{\ensuremath{\mathds{R}}}

\newcommand{\C}{\ensuremath{\mathds{C}}}

\newcommand{\ie}{\textit{i.e.}, }

\newcommand{\eg}{\textit{e.g.}, }

\newcommand{\mtr}{\hspace{-0.3mm}\ensuremath{^\top}}

\newcommand\raiseT[2]{\setbox0\hbox{$#1{#2}$}\raise\dp0\box0}

\newcommand{\V}{\ensuremath{\mathcal{V}}}

\newcommand{\E}{\ensuremath{\mathcal{E}}}

\newcommand{\M}{\ensuremath{\mathcal{M}}}

\newcommand{\etc}{\emph{etc.\ }}

\DeclareMathOperator{\trace}{\ensuremath{\mathrm{tr}}}

\newcommand{\Ni}{\ensuremath{\mathcal{N}_i}}

\makeatletter
\newcommand{\raisemath}[1]{\mathpalette{\raisem@th{#1}}}
\newcommand{\raisem@th}[3]{\raisebox{#1}{$#2#3$}}
\makeatother

\newcommand{\ts}[2][]{\ensuremath{\mathsf{T}_{#2}#1}}

\newcommand{\AGAS}{\textsc{agas}}

\definecolor{kthbluergb}{RGB}{25,84,166}%{0,108,183}
\definecolor{kthbluecmyk}{cmyk}{1,0.55,0,0}

\definecolor{kthblueA}{RGB}{25,84,166}%{0,108,183}
\definecolor{kthblueB}{RGB}{46,124,192}%{0,108,183}
\definecolor{kthblueC}{RGB}{112,153,209}%{0,108,183}
\definecolor{kthblueD}{RGB}{164,186,225}%{0,108,183}
\definecolor{kthblueE}{RGB}{211,220,241}%{0,108,183}

\newcounter{counter} % there is only one counter

\newtheorem{theorem}[counter]{Theorem}
\newtheorem{lemma}[counter]{Lemma}
\newtheorem{proposition}[counter]{Proposition}

\newtheorem{definition}[counter]{Definition}

\newtheorem{myalgorithm}[counter]{Algorithm}
\newtheorem{assumption}[counter]{Assumption}


\newcounter{parentnumber}

\makeatletter
\let\save@mathaccent\mathaccent
\newcommand*\if@single[3]{%
	\setbox0\hbox{${\mathaccent"0362{#1}}^H$}%
	\setbox2\hbox{${\mathaccent"0362{\kern0pt#1}}^H$}%
	\ifdim\ht0=\ht2 #3\else #2\fi
}
\newcommand*\rel@kern[1]{\kern#1\dimexpr\macc@kerna}
\newcommand*\widebar[1]{\@ifnextchar^{{\wide@bar{#1}{0}}}{\wide@bar{#1}{1}}}
\newcommand*\wide@bar[2]{\if@single{#1}{\wide@bar@{#1}{#2}{1}}{\wide@bar@{#1}{#2}{2}}}
\newcommand*\wide@bar@[3]{%
	\begingroup
	\def\mathaccent##1##2{%
		\let\mathaccent\save@mathaccent
		\if#32 \let\macc@nucleus\first@char \fi
		\setbox\z@\hbox{$\macc@style{\macc@nucleus}_{}$}%
		\setbox\tw@\hbox{$\macc@style{\macc@nucleus}{}_{}$}%
		\dimen@\wd\tw@
		\advance\dimen@-\wd\z@
		\divide\dimen@ 3
		\@tempdima\wd\tw@
		\advance\@tempdima-\scriptspace
		\divide\@tempdima 10
		\advance\dimen@-\@tempdima
		\ifdim\dimen@>\z@ \dimen@0pt\fi
		\rel@kern{0.6}\kern-\dimen@
		\if#31
		\overline{\rel@kern{-0.6}\kern\dimen@\macc@nucleus\rel@kern{0.4}\kern\dimen@}%
		\advance\dimen@0.4\dimexpr\macc@kerna
		\let\final@kern#2%
		\ifdim\dimen@<\z@ \let\final@kern1\fi
		\if\final@kern1 \kern-\dimen@\fi
		\else
		\overline{\rel@kern{-0.6}\kern\dimen@#1}%
		\fi
	}%
	\macc@depth\@ne
	\let\math@bgroup\@empty \let\math@egroup\macc@set@skewchar
	\mathsurround\z@ \frozen@everymath{\mathgroup\macc@group\relax}%
	\macc@set@skewchar\relax
	\let\mathaccentV\macc@nested@a
	\if#31
	\macc@nested@a\relax111{#1}%
	\else
	\def\gobble@till@marker##1\endmarker{}%
	\futurelet\first@char\gobble@till@marker#1\endmarker
	\ifcat\noexpand\first@char A\else
	\def\first@char{}%
	\fi
	\macc@nested@a\relax111{\first@char}%
	\fi
	\endgroup
}
\makeatother

\begin{document}

\maketitle
\thispagestyle{empty}
\pagestyle{empty}

%An important challenge for feedback control on nonlinear spaces is to ensure good performance on a global level. In particular, multi-agent systems such as the Kuramoto model are often multistable. The agents hence  risk getting stuck in undesired equilibria. This paper provides a consensus protocol on networks over hypersurfaces that stabilizes the consensus manifold almost globally provided a condition on the geometry of the manifold is met. Hypersurfaces include several nonlinear spaces of special interest, \eg the circle and the sphere. We show that the condition requires the hypersurface to be the boundary of a convex set. Moreover, we show that the condition can be simplified if the hypersurface is the level set of a strictly convex function. In particular, the condition is satisfied on all spheres except the circle, on ellipsoids whose parameters satisfy an additional requirement, and on hypersurfaces that are obtained from a certain type of small perturbations to a sphere.

%%%%%%%%%%%%%%%%%%%%%%%%%%%%%%%%%%%%%%%%%%%%%%%%%%%%%%%%%%%%%%%%%%%%%%%%%%%%%%%%
\begin{abstract}
% This is sort of bad. It does not talk about the geometric issues that the introduction addresses. It does not fit.

% There is a gap between

Consensus on nonlinear spaces is of use in many control applications. This paper proposes a gradient descent flow algorithm for consensus on hypersurfaces. We show that if an inequality holds, then the system converges for almost all initial conditions and all connected graphs. The inequality involves the hypersurface Gauss map and the gradient and Hessian of the implicit equation. Moreover, for the inequality to hold, it is necessary that the manifold is the boundary of a convex set. The literature already contains an algorithm for consensus on hypersurfaces. That algorithm on any ellipsoid is equivalent to our algorithm on the unit sphere. In particular, that algorithm achieves almost global synchronization on ellipsoids. These findings suggest that strong convergence results for consensus seeking gradient descent flows may be established on manifolds that are the boundaries of convex sets.

%These results suggests that consensus seeking gradient flow may converge almost globally on them on manifolds that are the boundaries of convex sets.

%Many multi-agent systems on networks over nonlinear spaces are known to be multistable, \eg the Kuramoto model. By contrast, the high-dimensional Kuramoto model on any connected network over the sphere synchronizes from almost all initial conditions. To better understand how the geometry and topology of a space influences the stability properties of multi-agent systems on networks over said space, we study synchronization on  hypersurfaces. Examples of hypersurfaces include the circle and the sphere. We derive a sufficient condition for the consensus manifold to be almost globally asymtoptically stable. We show that the condition requires the hypersurface to be the boundary of a convex set. Moreover, we show that the condition can be simplified if the hypersurface is the level set of a strictly convex function. We establish that the condition holds for some specific hypersurfaces such as all spheres and some ellipsoids.
%
\end{abstract}

%%%%%%%%%%%%%%%%%%%%%%%%%%%%%%%%%%%%%%%%%%%%%%%%%%%%%%%%%%%%%%%%%%%%%%%%%%%%%%%%

\section{Introduction}

\noindent Consensus on nonlinear spaces is of interest in many application  areas including robotics \cite{song2017intrinsic}, flocking \cite{al2018gradient}, opinion dynamics \cite{aydogdu2017opinion}, machine learning \cite{crnkic2018swarms}, and quantum synchronization \cite{ha2019second}. The problem of almost global consensus on nonlinear spaces is interesting from an applied point of view since it makes the probability of reaching consensus from a random initial condition independent of the number of agents. It is also interesting from a theoretical perspective since the global geometry and topology is what differentiates a Riemannian manifold $\M$ from Euclidean space $\R^m$. This paper explores how a consensus seeking gradient descent flow algorithm being almost globally convergent depends on the geometry and topology of the manifold it evolves on.

Consider a consensus seeking gradient descent flow of a disagreement function on a manifold $\M$ \cite{sarlette2009consensus}. Global convergence results are known for some special cases. For example, the consensus manifold $\mathcal{C}$ is almost globally asymptotically stable (\AGAS) for all connected networks over spheres of dimension $n\geq2$ \cite{markdahl2018tac}. There is hence at least one \AGAS{} consensus protocol on the boundary of every compact, star-shaped set for $n\geq2$, obtained by lifting the protocol on the corresponding $n$-sphere to $\M$. Another example; a necessary condition states that $\mathcal{C}$ cannot be \AGAS{} if $\M$ is simply connected \cite{markdahl2019kuramoto}. Since the boundary of any compact convex set of dimension $n$ in $\R^m$ is homeomorphic to the $n$-sphere in $\R^{n+1}$, the boundaries of compact convex sets are simply connected for $n\geq2$. Moreover, the boundary of any compact convex set can be described as a hypersurface in $\R^{n+1}$.

Working with arbitrary manifolds is difficult. Like spheres, hypersurfaces can be characterized by a single constraint, wherefore the methodology of \cite{markdahl2018tac,markdahl2020high} can be applied. For technical reasons we limit consideration to closed analytic hypersurfaces, \ie hypersurfaces that are analytic, compact and without boundaries. This paper provides a sufficient condition for $\mathcal{C}$ to be \AGAS{} for all connected networks over such hypersurfaces. The condition can only be satisfied if the hypersurface is the boundary of a convex set. However, because the condition is based on a quadratic Taylor expansion of the disagreement function, it cannot be used for manifolds on which the quadratic term vanishes. The question concerning the boundary of any convex set hence remains unresolved.

There is another algorithm for consensus on hypersurfaces in the literature \cite{zhu2013synchronization,zhu2014high,zhang2019synchronisation}. Compared to our algorithm, it is more restricted in terms of the hypersurfaces it can be applied to. We show that this algorithm on any ellipsoid is equivalent to our algorithm on the unit sphere.  Almost global convergence on ellipsoid have been established for graphs that are either complete or acyclic \cite{zhu2014high}. The result for our algorithm on networks over the sphere \cite{markdahl2018tac} also applies to the algorithm \cite{zhu2014high} on networks over ellipsoids, showing that it converges almost globally for all connected graphs.

\section{Preliminaries}

%\subsection{Hypersurfaces}

%\noindent Let $(\M,g)$ be a Riemannian manifold. The set $\M$ is a real, smooth manifold and the metric tensor $g_{\ve{x}}$ is an inner product on the tangent space $\ts[\M]{\ve{x}}$ at $\ve{x}$. %The map $x\mapsto g_{\ve{x}}(\ve{f},\ve{g})$ is smooth for any two differentiable vector fields $\ve{f},\ve{g}$ on $\M$. 
%Let $\M$ be embedded in an ambient Euclidean space, $\M\subset\R^n$. Take $g$ to be the standard Euclidean inner product $g_{\ve{x}}(\ve{y},\ve{z})=\langle\ve{y},\ve{z}\rangle=\vet{y}\ve{z}$.  Define the gradient of $f$ on $\M$ as $\Pi\nabla f$, where $\Pi:\R^n\rightarrow\ts[\M]{\ve{x}}$ is an orthogonal projection onto the tangent space of $\M$ and $\nabla$ is the usual gradient operator on $\R^n$. %Since $\widebar{f}$ is only used to define $\nabla f$, it does not matter which smooth extension we choose  \cite{loring2010an}.

\noindent The boundary $\partial\mathcal{S}$ of any compact convex set $\mathcal{S}\subset\R^m$, $\dim\partial\mathcal{S}=n<m$, can be transformed into a hypersurface in $\R^{n+1}$ by a change of coordinates. For technical reasons we focus on closed analytic hypersurfaces. A closed analytic hypersurface $\M\subset\R^{n+1}$ can without loss of generality be characterized as a set on the form
\begin{align*}
\M=\{\ve{y}\in\R^{n+1}\,|\,c(\ve{y})=0\},
\end{align*}
where $c:\R^{n+1}\rightarrow\R$ is an analytic function. The Jordan-Brouwer  theorem implies that $\M$ separates the set on which $c$ is positive from the set on which $c$ is negative \cite{lima1988jordan}. One of the sets is bounded while the other is unbounded. If the gradient $\nabla c(\ve{y})$ is considered as a vector located at $\ve{y}$, then it points towards the region on which $c$ is positive. This set can be assumed to be unbounded without loss of generality.

%The dimension of a manifold $\M$ equals that of its tangent space $\ts[\M]{\ve{y}}$. Since 
%%
%\begin{align*}
%\ts[\M]{\ve{y}}=\{\ve{z}\in\R^n\,|\,\langle\ve{z},\nabla c(\ve{y})\rangle=0\},
%\end{align*}
%%
%it follows that $\dim\ts[\M]{\ve{x}}=n-1$. This also implies that the codimension of $\M$ is $1$,  $\textrm{codim}\,\M=\dim\R^n-\dim\M=1$. We may hence refer to $\M$ as an hypersurface, \ie a generalization of a surface in $\R^3$ to $\R^n$ by analogue of the way that a hyperplane generalizes the notion of a plane.

A hypersurface $\M$ is called nonsingular if $\nabla c(\ve{y})\neq\ve{0}$ for all $\ve{y}\in\M$. Assume that $\M$ is nonsingular. Let 
\begin{align*}
\ve{n}(\ve{y})=\tfrac{\nabla c(\ve{y})}{\|\nabla c(\ve{y})\|}
\end{align*}
denote the unit normal obtained from the Gauss map $\ve{n}:\M\rightarrow\mathcal{S}^{n}$. The projection $\Pi:\R^n\rightarrow\ts[\M]{\ve{y}}$ on the tangent space of $\M$ at $\ve{y}$ is given by
\begin{align*}
\ve{z}\mapsto\left(\ma[n+1]{I}-\ve{n}(\ve{y})\vet{n}\!(\ve{y})\right)\ve{z},
\end{align*}
where the Gram-Schmidt rule $\ve{n}\perp\ve{z}-\langle\ve{z},\ve{n}\rangle\ve{n}$ is used to cancel the normal component of $\ve{z}$. This expression allows us to calculate the gradient $\Pi\nabla f(\ve{y})$ of $f(\ve{y})$ on $\M$ as
\begin{align*}
\Pi\nabla f(\ve{y})&=(\ma[n+1]{I}-\tfrac{\nabla c(\ve{y})}{\|\nabla c(\ve{y})\|}(\tfrac{\nabla c(\ve{y})}{\|\nabla c(\ve{y})\|})\mtr)\nabla f(\ve{y}),
\end{align*}
where $\nabla f(\ve{y})$ is the Euclidean gradient of $f(\ve{y})$ in $\R^{n+1}$.

%\begin{example}
%Examples of hypersurfaces on the desired form include the $n$-sphere in $\R^n$, the hyperboloid in $\R^3$, and the torus in $\R^3$:
%
%\begin{align*}
%\mathcal{S}^n&=\{\ve{x}\in\R^{n+1}\,|\,\|\ve{x}\|^2=1\},\\
%\mathcal{H}&=\{\ve{x}\in\R^3\,|\,(\tfrac{\ve[1]{x}}{a})^2+(\tfrac{\ve[2]{x}}{b})^2-(\tfrac{\ve[3]{x}}{c})^2=1\},\\
%\mathcal{T}&=\{\ve{x}\in\R^3\,|\,\Bigl(\sum_{i=1}^3\ve[i]{x}^2+R^2-r^2\Bigr)^2=4R^2(\ve[1]{x}^2+\ve[2]{x}^2)\}.
%\end{align*}
%
%\end{example}

%\subsection{Almost global asymptotic stability}

A set is almost globally asymptotically stable if almost all system trajectories converge to it:

%\noindent There are various definitions of an equilibrium set being almost globally asymptotically stable (\AGAS). The concept was introduced by Rantzer; in \cite{rantzer2001dual} it is implicitly defined in the context of a stability theorem. In \cite{monzon2003necessary} it is defined as a set to which almost all system trajectories converge. However, almost global convergence does not imply attractiveness. Moreover, an equilibrium set can be almost globally attractive without being stable \cite{freeman2013global}. In this paper we add asymptotical stability to almost global convergence.

\begin{definition}[\AGAS]\label{def:AGAS}
A Lyapunov stable equilibrium set $\mathcal{S}$ of a dynamical system $\vd{x}=\ve{f}(\ve{x})$ on a Riemannian manifold $(\M,g)$, where $\M\subset\R^{n+1}$, is said to be \AGAS{} if  $\lim_{\ve{x}\rightarrow\infty}\ve{x}(t)\in\mathcal{S}$ for all $\ve{x}(0)\in\M\backslash\mathcal{N}$, where $\mathcal{N}$ has Riemannian measure zero.
\end{definition}

%As a curiosity, note that an equilibrium set $\mathcal{S}$ being $\AGAS$ in the sense of Definition \ref{def:AGAS} does not necessarily imply that $\mathcal{S}$ is attractive since there could be an unstable equilibrium arbitrarily close to $\mathcal{S}$. This ambiguity can be removed by requiring that $\mathcal{S}$ is asymptotically stable. Indeed, the set that this is proved to be \AGAS{} in this paper is also asymptotically stable. 

%However, we would like to stress that we do not care about  $\mathcal{N}$. The risk of encountering a bad initial condition belonging to $\mathcal{N}$ is zero for any bounded probability distribution of initial conditions on $\M^N$. Questions about the location of $\mathcal{N}$ and the system behaviour on the region of attraction of $\mathcal{N}$ are therefore not important from a practical point of view.

\section{Distributed control design}

\noindent We use a graph $\mathcal{G}=(\mathcal{V},\mathcal{E})$ to model interactions between agents. Each node $i\in\V$ corresponds to an agent and each edge $\{i,j\}\in\mathcal{E}$ corresponds to a pair of communicating agents. The graph is assumed to be connected. Items associated with agent $i$ carry the subindex $i$; we denote the state of agent $i$ by $\ma[i]{x}\in\M$, the normal of $\M$ at $\ve[i]{x}$ by $\ve[i]{n}$, the projection onto the tangent space of $\M$ at $\ve[i]{x}$ by $\Pi_i$, the neighbor set of agent $i$ by $\mathcal{N}_i=\{j\in\V\,|\,\{i,j\}\in\E\}$, the Euclidean gradient of $V$ with respect to $\ve[i]{x}$ by $\nabla_i V$ \etc We call $\ve{x}=(\ve[i]{x})_{i=1}^N\in\M^N$ a configuration of agents.

Consider a dynamical system defined on $\M^N$. The dynamics of agent $i$ could \eg be
\begin{align}
\vd[i]{x}=\ve[i]{u}\label{eq:xd1}
\end{align}
where $\ve[i]{u}\in\ts[\M]{i}$ is the control signal. Another option is 
\begin{align}
\vd[i]{x}=\Pi_i\ve[i]{u}=(\ma[n+1]{I}-\ve[i]{n}\vet[i]{n})\ve[i]{u}\label{eq:xd2}
\end{align}
where $\ve[i]{u}\in\R^{n+1}$ and $\ve[i]{n}=\ve{n}(\ve[i]{x})$ is introduced for the sake of notational convenience. Note that the right-hand sides of \eqref{eq:xd1} and \eqref{eq:xd2} belong to $\ts[\M]{\ve[i]{x}}$. Suppose that $\ve[i]{x}(0)\in\M$, that $\ve[i]{u}$ is Lipschitz, and that $\M$ is a $C^2$ manifold. Then $\ve[i]{x}(t)\in\M$ for all $t\in\R$ by the Bony-Brezis theorem since $\langle\ve{v},\nabla_i c(\ve[i]{x})\rangle=0$ for all $\ve{v}\in\ts[\M]{\ve[i]{x}}$ and all $\ve[i]{x}\in\M$. By confining $\vd[i]{x}$ to $\ts[\M]{i}$, we confine $\ve[i]{x}$ to $\M$. %This implies that  constraint $\ve[i]{x}\in\mathcal{S}^2$ holds for all possible Lipschitz actions $\ve[i]{u}\in\R^n$ taken by the fish in our interpretation of system \eqref{eq:xd2}. Also note that \eqref{eq:xd2} is well defined with respect to scalings of $c$ 

The input model \eqref{eq:xd1} corresponds to a situation where the constraint $\ve[i]{x}\in\M$ is adopted to accomplish a task whereas the model \eqref{eq:xd2} refers to the case where the mechanical design of a systems constrains it to only be actuated in a certain fashion. An example of \eqref{eq:xd1} is a team of satellites in orbit; they could leave the orbit if so desired. Examples of \eqref{eq:xd2} include camera sensor networks where each camera is mounted on a spherical joint. The orientation of camera $i$ is always some $\ve[i]{x}\in\mathcal{S}^2$ regardless of the control input.

The goal of consensus seeking systems is for the agents to asymptotically approach the consensus manifold
\begin{align}
\mathcal{C}=\{(\ve{x})_{i=1}^N\in\M^N\}.\label{eq:C}
\end{align}
The set $\mathcal{C}$ is a manifold $\mathcal{C}\simeq\M$ by the diffeomorphism $\M^N\rightarrow\mathcal{C}:(\ve{x})_{i=1}^N\mapsto\ve{x}$. If the agents are satellites in orbit that satisfy $\ve[i]{x}\in\mathcal{S}^2$, then this would be interpreted as all $N$ agents meeting up at one point. If the agents are rigid bodies whose pointing direction (reduced attitude) is modelled as $\ve[i]{x}\in\mathcal{S}^2$, then a consensus implies that all $N$ bodies are pointing in the same direction. 

As a measure of the distance to consensus, consider the disagreement function $V:\M^N\rightarrow\R$ given by
\begin{align}\label{eq:V}
V(\ve{x})=\tfrac12\sum_{i,j\in\E}a_{ij}\|\ve[j]{x}-\ve[i]{x}\|^2,
\end{align}
where $a_{ij}\in[0,\infty)$.
Clearly, $V=0$ if and only if $\ve{x}=(\ve[i]{x})_{i=1}^N\in\mathcal{C}$, \ie no disagreement. The consensus seeking algorithm that we study in this paper is the gradient descent flow of \eqref{eq:V}. The gradient of $V$ on $\M$ is given by
\begin{align*}
%\nabla_iV&=\sum_{j\in\Ni} \ve[i]{x}-\ve[j]{x},\\
\Pi_i\nabla_iV&=\left(\ma[n]{I}-\ve[i]{n}\vet[i]{n}\right)\sum_{j\in\Ni} a_{ij}(\ve[i]{x}-\ve[j]{x}).
\end{align*}

We are now ready to state the main algorithm of this paper. This algorithm previously appears in \cite{sarlette2009consensus}, although their work is limited to the case when the norm of the states are constant, $\|\ve[i]{x}\|=k$, \ie the case when $\M$ is a sphere. Moreover, they only show local stability results.
\begin{myalgorithm}\label{algo:main}
The consensus seeking gradient descent flow on $\M$ is given by
\begin{align}
\vd[i]{x}&=-\Pi_i\nabla_iV=\left(\ma{I}-\ve[i]{n}\vet[i]{n}\right)\sum_{j\in\Ni}a_{ij}( \ve[j]{x}-\ve[i]{x}).\label{eq:flow}
\end{align}
\end{myalgorithm}

Suppose that $\M$ is closed, then the solution $\ve{x}(t)=(\ve[i]{x})_{i=1}^N$ to \eqref{eq:flow} is unique and exists for all $t\in\R$ \cite{khalil2002nonlinear}.

There is another algorithm for consensus on hypersurfaces in the literature:
\begin{myalgorithm}[Zhu \cite{zhu2014high}]\label{algo:Zhu}
The consensus seeking algorithm on $\M$ is given by
\begin{align}
\vd[i]{x}&=\left(\ma{I}-\tfrac{\ve[i]{x}\nabla c(\ve[i]{x})\mtr}{\langle\ve[i]{x},\nabla c(\ve[i]{x})\rangle}\right)\sum_{j\in\Ni}a_{ij}( \ve[j]{x}-\ve[i]{x})\nonumber\\
&=\left(\ma{I}-\tfrac{\ve[i]{x}\nabla c(\ve[i]{x})\mtr}{\langle\ve[i]{x},\nabla c(\ve[i]{x})\rangle}\right)\sum_{j\in\Ni}a_{ij}\ve[j]{x}\label{eq:zhu}
\end{align}
\end{myalgorithm}
\vspace{2mm}

To briefly compare Algorithm \ref{algo:main} and \ref{algo:Zhu}, note that Algorithm \ref{algo:main} requires that $\nabla c(\ve{x})\neq0$ on $\M$ whereas Algorithm \ref{algo:Zhu} also requires $\langle\ve{x},\nabla c(\ve{x})\rangle\neq0$. The two algorithms are identical when $\nabla c(\ve{x})=k\ve{x}$ for some $k\in\R$, \ie when $\M$ is $\mathcal{S}^n$. Indeed, both algorithms are conceived of as generalizations of a consensus algorithm on the $n$-sphere \cite{olfati2006swarms,lohe2010quantum}. In general, it may be more difficult to establish convergence of Algorithm \ref{algo:Zhu} since it is not a gradient descent flow. We provide the following result about Algorithm \ref{algo:Zhu}:

\begin{proposition}
The system \eqref{eq:zhu} on an ellipsoid is equivalent to the system \eqref{eq:flow} on the unit sphere.
\end{proposition}

\begin{proof}
Let $\M$ be an ellipsoid, \ie
\begin{align*}
\M=\{\ve{y}\in\R^n\,|\,c(\ve{y})=\tfrac12\langle\ve{y},\ma{A}\ve{y}\rangle-1=0\},
\end{align*}
where $\ma{A}$ is a positive definite matrix. The dynamics \eqref{eq:zhu} of the consensus seeking system on $\M$ under Algorithm \ref{algo:Zhu} is
\begin{align*}
\vd[i]{y}=(\ma{I}-\tfrac{1}{\langle\ve[i]{y},\ma{A}\ve[i]{y}\rangle}\ve[i]{y}\vet[i]{y}\ma{A})\sum_{\{i,j\}\in\E}a_{ij}(\ve[j]{y}-\ve[i]{y}).
\end{align*}
Let $\ma{L}$ denote the Cholesky factor of $\ma{A}$, \ie $\ma{A}=\ma{L}\mat{L}$. Introduce $\ma[i]{z}=\mat{L}\ve[i]{y}$ and note that $\|\ma[i]{z}\|^2=\langle\ve[i]{y},\ma{A}\ve[i]{y}\rangle=1$, \ie $\ve[i]{z}\in\mathcal{S}^n$. Calculate
\begin{align*}
\vd[i]{z}&=\mat{L}(\ma{I}-\tfrac{1}{\langle\ve[i]{y} ,\ma{L}\mat{L}\ve[i]{y}\rangle}\ve[i]{y}\vet[i]{y}\ma{L}\mat{L})\sum_{j\in\Ni}a_{ij}(\ve[j]{y}-\ve[i]{y})\\
&=(\ma{I}-\tfrac{\ve[i]{z}}{\|\ve[i]{z} \|}(\tfrac{\ve[i]{z}}{\|\ve[i]{z}\|})\mtr)\sum_{j\in\Ni}a_{ij}(\ve[j]{z}-\ve[i]{z})\\
&=(\ma{I}-\ve[i]{z}\vet[i]{z})\sum_{j\in\Ni}a_{ij}\ve[j]{z},
\end{align*}
which is the system \eqref{eq:flow} on the unit sphere.
\end{proof}

%Consider the input model \eqref{eq:xd1}. Suppose that agent $i$ can sense $\Pi_i\sum_{j\in\Ni} \ve[j]{x}-\ve[i]{x}$. Then \eqref{eq:flow} can be implemented in a distributed fashion. Second, consider the input model \eqref{eq:xd2}. Suppose that agent $i$ can sense $\sum_{j\in\Ni} \ve[i]{x}-\ve[j]{x}$ or $\Pi_i\sum_{j\in\Ni} \ve[j]{x}-\ve[i]{x}$. Then \eqref{eq:flow} can be implemented in a distributed fashion. This is due to projections being idempotent. It follows that feedback under the actuation model \eqref{eq:xd1} requires at least as much sensing capabilities as feedback under the model \eqref{eq:xd2}. The model \eqref{eq:xd2} could hence be preferred. However, the choice of \eqref{eq:xd1} or \eqref{eq:xd2} is immaterial to the rest of this paper, which concerns the closed loop system \eqref{eq:flow}. In practice, the question of what an agent can or cannot sense has to be resolved on a case by case basis for each application. 

\section{Almost global asymptotic stability}

\noindent The main result of this paper states that for any closed, analytic manifold that satisfies a geometric condition, the consensus manifold $\mathcal{C}$ is an \AGAS{} equilibrium manifold of the gradient descent flow \eqref{eq:flow}. In the derivation of the main result, the condition appears as an expression which relates the relative information $\ve[j]{x}-\ve[i]{x}$ for any $\{i,j\}\in\E$ at an equilibria of the system to some geometric quantities evaluated at $\ve[i]{x}$ and $\ve[j]{x}$. However, it is difficult to say which pairs of points are part of an equilibrium and which are not. As such, we make the conservative requirement that the condition is satisfied at any pair of points $\ve{y},\,\ve{z}\in\M$.

Let $\mathcal{Q}$ denote the set of all equilibria of the gradient descent flow \eqref{eq:flow} that does not belong to the consensus manifold  $\mathcal{C}$ given by \eqref{eq:C}. Most of this section is concerned with establishing that each equilibria in $\mathcal{Q}$ is unstable; a result which is summarized in Proposition \ref{prop:eigenvalues}. This leads us to sufficient conditions for $\mathcal{C}$ to be an \AGAS{} set of equilibria of the gradient descent  flow \eqref{eq:flow}. Before that we establish Proposition \ref{prop:local} which shows that the consensus manifold $\mathcal{C}$ given by \eqref{eq:C} is asymptotically stable as a set. Note that Proposition \ref{prop:local} only requires $\M$ to be a closed analytic manifold, \ie a compact analytic manifold without boundary.

\subsection{Local stability}

\begin{proposition}\label{prop:local}
Let $\M\subset\R^m$ be a closed, analytic, embedded Riemannian manifold. The consensus manifold $\mathcal{C}=\{(\ve{x})_{i=1}^N\in\M^N\}$ is an asymptotically stable equilibrium set of the gradient descent flow $\vd{x}=-(\Pi_i\nabla_i V)_{i=1}^N$ and
\begin{align*}
V(\ve{x})&=\tfrac12\sum_{\{i,j\}\in\E}\|\ve[j]{x}-\ve[i]{x}\|^2.
\end{align*}
\end{proposition}

\begin{proof}
The potential function of a gradient descent flow decreases with time, 
\begin{align}\label{eq:Vd}
\dot{V}=\langle\Pi_i\nabla V,\vd{x}\rangle=-\|\Pi_i\nabla V\|^2.
\end{align}
Since $V\geq 0$ with $V=0$ if and only if $\ve{x}\in\mathcal{C}$, we can take $V$ as a Lyapunov function and conclude that $\mathcal{C}$ is stable.

Since $\M$ is closed, the gradient descent flow converges to a connected component of the set of critical points of $V$ \cite{helmke2012optimization}. By \eqref{eq:Vd}, any sublevel set of $V$ is forward invariant. Moreover, all sublevel sets contain $\mathcal{C}$. If there is an open sublevel set of $V$ which does not intersect $\mathcal{Q}$, then there is an open neighborhood of $\mathcal{C}$ from which $\ve{x}$ converges to $\mathcal{C}$.

Since $V$ is analytic it satisfies the \L{}ojasiewicz inequality on Riemannian manifolds \cite{kurdyka2000proof}. For every $\ve{x}\in\mathcal{C}$ there is an open ball $\mathcal{B}(\ve{x})$, an $\alpha<1$, and a $k>0$ such that 
\begin{align*}
V(\ve{y})^\alpha\leq k\|\Pi \nabla V(\ve{y})\|
\end{align*}
for all $\ve{y}\in\mathcal{B}(\ve{x})$. If $\ve{y}\in\mathcal{Q}$, then $\Pi\nabla V(\ve{y})=\ve{0}$ whereby $V(\ve{y})=0$, which implies  $\ve{y}\in\mathcal{C}$, a contradiction. Hence $\mathcal{Q}\cap\mathcal{B}(\ve{x})=\emptyset$.

Consider the value of $q=\inf_{\ve{x}\in\mathcal{Q}}V(\ve{x})$. If $q=0$, then there is a sequence $\{\ve[k]{x}\}_{k=1}^\infty$ such that $\lim_{k\rightarrow\infty}V(\ve[k]{x})=0$. Since $\M$ is a closed manifold, the sequence $\{\ve[k]{x}\}_{k=1}^\infty$ has a subsequence which converges to some $\ve{y}\in\M$. Moreover, $V(\ve{y})=0$ whereby $\ve{y}\in\mathcal{C}$. For each $\varepsilon>0$ there must be a $\ve{z}(\varepsilon)\in\mathcal{Q}$ (an element of the subsequence) such that $\|\ve{y}-\ve{z}(\varepsilon)\|<\varepsilon$. This contradicts  $\mathcal{Q}\cap\mathcal{B}(\ve{y})=\emptyset$. Hence $q>0$ and all trajectories that start in the level set $\{\ve{x}\in\M\,|\,V(\ve{x})<q\}$ converges to $\mathcal{C}$. \end{proof}

%Form a neighborhood of $\mathcal{C}$, 
%
%\begin{align*}
%\mathcal{B}(\mathcal{C})=\cup_{\ve{x}\in\mathcal{C}}\mathcal{B}(\ve{x}).
%\end{align*}
%
%The boundary $\partial\mathcal{B}$ of $\mathcal{B}$ is a compact set. The potential function $V$ attains its minimal value on $\partial\mathcal{B}$ by the extreme value theorem. If $\min_{\ve{x}\in\partial\mathcal{B}(\mathcal{C})}V(\ve{x})>0$, then we have created an invariant set which contains $\mathcal{C}$ but does not intersect $\mathcal{Q}$. If $\min_{\ve{x}\in\partial\mathcal{B}\mathcal{C}}V(\ve{x})=0$, then this value is achieved at a point $\ve{z}\in\mathcal{C}$. However, that   $\ve{z}\in\mathcal{B}(\ve{z})$ and $\ve{z}\in\partial\mathcal{B}(\ve{z})$   contradicts $\mathcal{B}(\ve{z})$ being an open set.\end{proof}

%I WANTED TO MAKE A GEOMETRIC PROOF BUT IT IS STRANGELY DIFFICULT.

This result is similar to Proposition 7 in \cite{sarlette2009consensus}. Note however that our proof of local stability only makes use of the properties of gradient descent flows of analytic function on closed manifolds. To show that $\mathcal{C}$ is  \AGAS{} we also need to consider the geometry and topology of $\M$. In particular, $\mathcal{C}$ is not an \AGAS{} equilibrium manifold of \eqref{eq:flow} if $\M$ is a multiply connected hypersurface such as a circle or a torus \cite{markdahl2019kuramoto}. The sufficient condition for \AGAS{} established in this paper places requirements on $\M$ that exclude such cases.

\subsection{Main result}

\noindent Our main result is establishes almost global convergence to the consensus manifold if the following assumption holds:
\begin{assumption}\label{ass:ineq}
	Suppose $c$ satisfies
	\begin{align*}
	\langle\ve{n}(\ve{y}),\ve{n}(\ve{z})\rangle^2+\tfrac{\langle \ve{y}-\ve{z},\nabla c(\ve{y})\rangle(\Delta c(\ve{y})-\langle\ve{n}(\ve{y}),\nabla^2c(\ve{y})\ve{n}(\ve{y})\rangle)}{\|\nabla c(\ve{y})\|^2}\geq1,
	\end{align*}
	for all $\ve{y},\ve{z}\in\M$ and with equality only if $\ve{y}=\ve{z}$, where $\ve{n}:\M\rightarrow\mathcal{S}^{n-1}$ is the Gauss map and $\Delta$ is the Laplace-Beltrami operator, $\Delta c(\ve{y})=\trace\nabla^2c(\ve{y})$.
\end{assumption}

\begin{theorem}\label{th:main}
	Let $c$ be a real analytic function that satisfies Assumption \ref{ass:ineq}. The consensus manifold is an \AGAS{} equilibrium set of the gradient descent flow
	\begin{align*}
	V&=\tfrac12\sum_{\{i,j\}\in\E}a_{ij}\|\ve[i]{x}-\ve[j]{x}\|^2,\\
	\vd[i]{x}&=-\nabla_i V\\
	&=(\ma{I}-\tfrac{\nabla c(\ve[i]{x})}{\|\nabla c(\ve[i]{x})\|}(\tfrac{\nabla c(\ve[i]{x})}{\|\nabla c(\ve[i]{x})\|})\mtr)\!\sum_{\{i,j\}\in\E}\!a_{ij}(\ve[j]{x}-\ve[i]{x}),
	\end{align*}
	on the $N$-fold product of $\M=\{\ve{y}\in\R^n\,|\,c(\ve{y})=0\}$.
	%
	%Moreover, the consensus manifold is asymptotically stable.
\end{theorem}

Since the proof of the main result is somewhat long, we have broken it into parts. First, we need two definitions.

\begin{definition}
Let $\Sigma$ be a dynamical system on $\mathcal{S}\subset\R^{n+1}$ whose solution $\ma{\Phi}(t;\ve{x})$, $\ma{\Phi}(0;\ve{x})=\ve{x}$ exists for all $t\in\R$ and all $\ve{x}\in\mathcal{S}$. The system $\Sigma$ is said to be pointwise convergent if for each $\ve{x}\in\mathcal{S}$ there is exactly one $\omega$-limit point $\lim_{i\rightarrow\infty}\ma{\Phi}(t_i;\ve{x})$ for all $(t_i)_{i=1}^\infty$ such that  $\lim_{i\rightarrow\infty}t_i=\infty$.
\end{definition}

\begin{definition}
An equilibrium point $\ve{y}\in\R^{n+1}$ of a dynamical system $\vd{x}=\ve{f}(\ve{x})$ is said to be exponentially unstable if the Jacobian matrix of $\ve{f}(\ve{x})$ evaluated at $\ve{y}$ has a strictly positive eigenvalue.
\end{definition}

For pointwise convergent systems, any set of exponentially unstable equilibria have a region of attraction with Riemannian measure zero \cite{freeman2013global}. The system \eqref{eq:flow} is pointwise convergent due to being a gradient descent flow of an analytic function on an analytic manifold  \cite{lageman2007convergence}. The problem of establishing almost global convergence has hence been reduced to showing that all equilibria besides those belonging to the consensus manifold are exponentially unstable.

\subsection{Positive eigenvalues}

%We do not have an intuitive interpretation of Assumption \ref{ass:ineq}, but we show in Lemma \ref{le:convex} that Assumption \ref{ass:ineq} is satisfied under another, stronger assumption on the convexity of $c$.

\noindent Let $\ma{L}(\ve{x})\in\R^{N(n+1)\times N(n+1)}$ denote the linearization matrix of the gradient descent flow \eqref{eq:flow} at the point $\ve{x}\in\M^N$. Our aim is to show that the eigenvalues of $\ma{L}(\ve{x})$ are positive for all equilibria $\ve{x}\notin\mathcal{C}$. Note that $\ma{L}(\ve{x})$ is related to the Hessian matrix $\ma{H}(\ve{x})\in\R^{Nn\times Nn}$ of $V$ as $\ma{L}(\ve{x})=-\ma{H}(\ve{x})$ \cite{helmke2012optimization}.

\begin{proposition}\label{prop:eigenvalues}
Let $\M\subset\R^n$ be a hypersurface for which Assumption \ref{ass:ineq} holds. The eigenvalues of the linearization matrix of the gradient descent flow \eqref{eq:flow} have strictly negative real parts at any equilibrium point except for those belonging to the consensus manifold $\mathcal{C}$.
\end{proposition}

\begin{proof}
The equilibria of the gradient descent flow \eqref{eq:flow} are critical points of the optimization problem
\begin{align}\label{eq:min}
\min_{\ve{x}\in\M^N}V(\ve{x})=\tfrac12\sum_{\{i,j\}\in\E}a_{ij}\|\ve[j]{x}-\ve[i]{x}\|^2.
\end{align}
We will analyze \eqref{eq:flow} in an optimization framework, making use of the associated techniques and terminology. Our approach is based on the Lagrange conditions for optimality in equality constrained nonlinear programming  \cite{nocedal1999numerical}.

%The Jacobian of the gradient descent flow \eqref{eq:flow} is the negative of the Hessian matrix of the potential function $V$ \eqref{eq:V}. 

Introduce the Lagrangian $\mathcal{L}:\M^N\times\R\rightarrow\R$ given by
\begin{align*}
\mathcal{L}(\ve{x},\ve{\lambda})&=V+\sum_{i\in\V}\lambda_i c(\ve[i]{x})\\
&=\tfrac12\sum_{\{i,j\}\in\mathcal{E}}a_{ij}\|\ve[i]{x}-\ve[j]{x}\|^2+ \sum_{i\in\V}\lambda_i c(\ve[i]{x}),
\end{align*}
where $\ve{\lambda}=[\lambda_i]\in\R^N$. The optimal solutions to \eqref{eq:min} are critical points of $\mathcal{L}$. The critical points of $\mathcal{L}$ are exactly the  eigenvalues of $\eqref{eq:flow}$. Calculate the Euclidean gradient of  $\mathcal{L}$,
\begin{align*}
\nabla_i\mathcal{L}&=\sum_{j\in\Ni}a_{ij}(\ve[i]{x}-\ve[j]{x})+\lambda_i\nabla_i c(\ve[i]{x}),\\
\tfrac{\partial}{\partial\lambda_i}\mathcal{L}&=c(\ve[i]{x}).
\end{align*}
The Hessian of $\mathcal{L}$ with respect to $\ve[i]{x}$, $\ve[k]{x}$ is a $N(n+1)\times N(n+1)$ block matrix $\nabla^2\mathcal{L}$, with the $ki$ block given by
\begin{align*}
%\tfrac{\partial^2{\mathcal{L}}}{\partial \ve[k]{x}\partial \ve[i]{x}}
(\nabla^2\mathcal{L})_{ki}=\begin{cases}\sum_{j\in\Ni}a_{ij}\ma[n+1]{I}+\lambda_i\nabla_i^2 c(\ve[i]{x}) & \textrm{ if }k=i,\\
-a_{ik}\ma[n+1]{I} & \textrm{ if }k\in\Ni,\\
\ma{0} & \textrm{ otherwise.}
\end{cases}
\end{align*}

The nullspace $\ker\nabla c_i$ of the constraint gradients is the image set of the symmetric matrix
\begin{align*}
\ma[i]{Z}=\ma[n+1]{I}-\ve[i]{n}\vet[i]{n},%\frac{1}{\|\nabla c_i(\ve[i]{x})\|^2}\nabla_i c(\ve[i]{x})\nabla_i c(\ve[i]{x})^T.
\end{align*}
where $\ve[i]{n}=\ve{n}(\ve[i]{x})$ and $\ve{n}$ is the Gauss map. Let $\ma{Z}$ denote the blockdiagonal matrix with $\ma[i]{Z}$
as the $ii$ block. Form the matrix $\ma{H}(\ve{x})=\ma{Z}\nabla^2\mathcal{L}\ma{Z}$ whose $ki$ block is 
\begin{align*}
\ma[k]{Z}\tfrac{\partial^2\mathcal{L}}{\partial \ve[k]{x}\partial \ve[i]{x}}\ma[i]{Z}=\begin{cases}\sum_{j\in\Ni}a_{ij}\ma[i]{Z}+\lambda_i\ma[i]{Z}\nabla_i^2 c\ma[i]{Z} & \textrm{ if }k=i,\\
-a_{ki}\ma[k]{Z}\ma[i]{Z} & \textrm{ if }k\in\Ni,\\
\ma{0} & \textrm{ otherwise},
\end{cases}
\end{align*}
where we used that $\ma[i]{Z}^2=\ma[i]{Z}$, which follows from  $\ma[i]{Z}$ being a projection matrix. 

Let $\mathsf{T}\M^N$ denote the tangent bundle of $\M^N$,
\begin{align*}
	\mathsf{T}\M^N=\{(\ve{x},\ve{v})\,|\,\ve{x}\in\M^N,\ve{v}\in\ts[\M^N]{\ve{x}}\}.
\end{align*}
The matrix $\ma{H}(\ve{x})$ is the Riemannian Hessian operator $\ma{H}(\ve{x}):\mathsf{T}\M^N\rightarrow\ts[\M]{\ve{x}}:(\ve{x},\ve{v})\mapsto\ma{H}(\ve{x})\ve{v}$ of $V$ on $\M$ \cite{birtea2015hessian}. It also appears in the necessary second order optimality conditions for equality constrained problems, with $\ma{H}(\ve{x})$ being positive semi-definite on $\mathsf{T}\M^N$ if $\ve{x}$ is an optimal solution to \eqref{eq:min} that satisfies some additional requirements \cite{nocedal1999numerical}.

Let $\ma{L}(\ve{x})=-\ma{H}(\ve{x})$ be the linearization matrix of the gradient descent flow \cite{helmke2012optimization}. Note that $\ma{L}$ is symmetric wherefore its field of values
\begin{align*} W(\ma{L})=\{\langle\ve{v},\ma{L}\ve{v}\rangle\,|\,\ve{v}\in\C^{Nn}\}=\{\langle\ve{v},\ma{L}\ve{v}\rangle\,|\,\ve{v}\in\R^{Nn}\}
\end{align*}
is real. Consider the Rayleigh quotient  $R:\mathsf{T}\M^N\rightarrow\R$ given by $R(\ve{x},\ve{v})=\langle\ve{v},\ma{L}(\ve{x})\ve{v}\rangle/\langle\ve{v},\ve{v}\rangle$. Let $\alpha(\ve{x})$ denote the spectral abscissa of $\ma{L}(\ve{x})$,  
\begin{align*}
\alpha(\ve{x})=\max_{\ve{v}\in\ts[\M^N]{\ve{x}}} R(\ve{x},\ve{v}).
\end{align*}

Since $\ma{L}(\ve{x})$ is symmetric, $\alpha(\ve{x})$ equals the largest eigenvalue of $\ma{L}(\ve{x})$. It is bounded below as $ \alpha(\ve{x})\geq R(\ve{x},\ve{v})$  for all $\ve{v}\in\ts[\M^N]{\ve{x}}$ by the min-max theorem. It follows that $-R(\ve{x},\ve{v})$ is an upper bound on the smallest eigenvalue of $\ma{H}(\ve{x})$. If $R(\ve{x},\ve{v})$ assumes a positive value for some argument, then the $\ma{H}(\ve{x})$ cannot be positive definite and the necessary optimality conditions fails to hold.

To obtain a lower bound for $\alpha(\ve{x})$, consider the tangent vector $\ve{v}=[\Pi_1\ve{u}\ldots\Pi_N\ve{u}]=[\ma[1]{Z}\ve{u}\ldots\ma[N]{Z}\ve{u}]$ for any $\ve{u}\in\R^n$ such that $\|\ve{v}\|=1$. The intuition for this step is that all agents are located at some equilibrium $\ve{x}$ and that we perturb all of them in the same direction, \ie towards the consensus manifold. Because all agents move towards the same region of the consensus manifold, it is possible that cohesion is increased whereby $V$ decreases. We calculate the effect this has on the quadratic term in the Taylor expanasion of $V$, \ie the term that depends on $\ma{H}(\ve{x})=-\ma{L}(\ve{x})$. The contribution of the linear term is zero due to $\nabla V=\ve{0}$ at any equilibrium. 

Calculate the Rayleigh quotient,
\begin{align*}
R(\ve{x},\ve{v})%&=\sum_{i\in\V}\langle\ve{w},\ma[ii]{H}\ve{w}\rangle+\sum_{j\in\Ni}\langle\ve{w},\ma[ji]{H}\ve{w}\rangle\\
&=\sum_{i\in\V}\langle\ve{u},(\ma[ii]{L}(\ve{x})+\sum_{j\in\Ni}\ma[ij]{L})\ve{u}\rangle\\
&=\sum_{i\in\V}\langle\ve{u},-(\lambda_i\ma[i]{Z}\nabla_i^2c\ma[i]{Z}+\!\!\sum_{j\in\Ni}a_{ij}(\ma[i]{Z}-\ma[i]{Z}\ma[j]{Z}))\ve{u}\rangle.
\end{align*}
Denote
\begin{align*}
\ma{M}(\ve{x})&=-\sum_{i\in\V}\lambda_i\ma[i]{Z}\nabla_i^2c\ma[i]{Z}+\sum_{j\in\Ni}a_{ij}(\ma[i]{Z}-\ma[i]{Z}\ma[j]{Z}).
\end{align*}
Hence $R(\ve{x},\ve{v})=\langle\ve{u},\ma{M}\ve{u}\rangle$. Let $(\mu_i,\ve[i]{u})$, where $\|\ve[i]{u}\|=1$, denote the eigenpairs of $\ma{M}$. Take $\ve{u}=\sum_{i=1}^{n+1}\ve[i]{u}/\|\sum_{i=1}^{n+1}\ve[i]{u}\|$ whereby $\ve{v}=\ma{Z}\ve{u}$ satisfies $\ve{v}\in\ts[\M]{\ve[i]{x}}$, $\|\ve{v}\|=1$ (note that $\sum_{i=1}^{n+1}\ve[i]{u}\neq\ve{0}$ by linear independence). Then 
\begin{align*}
R(\ve{x},\ve{v})=\frac{\trace\ma{M}(\ve{x})}{\|\sum_{i=1}^n\ve[i]{u}\|}.
\end{align*}
It remains to show that $\trace\ma{M}(\ve{x})\geq0$.

For the sake of notational convenience, write
\begin{align*}
\ma[i]{N}=\ve[i]{n}\vet[i]{n}=\tfrac{\nabla_i c(\ve[i]{x})}{\|\nabla_i c(\ve[i]{x})\|}(\tfrac{\nabla_i c(\ve[i]{x})}{\|\nabla_i c(\ve[i]{x})\|})\mtr.
\end{align*}
whereby $\ma[i]{Z}=\ma[n]{I}-\ma[i]{N}$. Rewrite
\begin{align*}
\ma{M}={}&-\sum_{i\in\V}\lambda_i(\ma[n]{I}-\ma[i]{N})\nabla_i^2c(\ma[n]{I}-\ma[i]{N})+\\
&\sum_{j\in\Ni}a_{ij}(\ma[n]{I}-\ma[i]{N}-(\ma[n]{I}-\ma[i]{N})(\ma[n]{I}-\ma[j]{N}))\\
={}&-\sum_{i\in\V}\lambda_i(\nabla_i^2c-\ma[i]{N}\nabla_i^2c-\nabla_i^2c\ma[i]{N}+\ma[i]{N}\nabla_i^2c\ma[i]{N})+\\
&\sum_{j\in\Ni}a_{ij}(\ma[j]{N}-\ma[i]{N}\ma[j]{N}).
\end{align*}

Solve $\nabla_i\mathcal{L}=\ve{0}$ for
\begin{align*}
\lambda_i=\tfrac{1}{\|\nabla_i c\|^2}\sum_{j\in\Ni}a_{ij}\langle\nabla_i c(\ve[i]{x}),\ve[j]{x}-\ve[i]{x}\rangle.
\end{align*}
Note that $\lambda_i$ is well-defined since $\M$ is nonsingular by assumption, which implies $\nabla_i c\neq\ve{0}$. Note that $\trace\ma[i]{N}=1$. Let $\Delta_i c= \trace\nabla_i^2c(\ve[i]{x})$ denote the Laplace-Beltrami operator acting on $c$. Calculate
\begin{align*}
\trace\ma{M}={}&-\sum_{i\in\V}\lambda_i(\Delta_ic-\langle\ve[i]{n},\nabla_i^2c\ve[i]{n}\rangle)+\\
&\sum_{j\in\Ni}a_{ij}(1-\langle\ve[i]{n},\ve[j]{n}\rangle^2)\\
={}&\sum_{i\in\V}\sum_{j\in\Ni}a_{ij}[-1+\langle\ve[i]{n},\ve[j]{n}\rangle^2+\\
&\tfrac{1}{\|\nabla_ic\|^2}\langle \ve[i]{x}-\ve[j]{x},\nabla_i c\rangle(\Delta_ic-\langle\ve[i]{n},\nabla_i^2c\ve[i]{n}\rangle)]
\end{align*}

The sum in the expression for $\trace\ma{M}$ is positive if every term is positive, \ie if
\begin{align*}
\langle\ve[i]{n},\ve[j]{n}\rangle^2+\tfrac{\langle\ve[i]{x}-\ve[j]{x},\nabla_i c\rangle(\Delta_ic-\langle\ve[i]{n},\nabla_i^2c\ve[i]{n}\rangle)}{\|\nabla_ic\|^2}\geq1
\end{align*}
with equality only when $\ve[i]{x}=\ve[j]{x}$. This relation holds by Assumption \ref{ass:ineq} on the geometry of $\M$.\end{proof}

\section{Convexity}

\subsection{Convex sets}

\noindent Assumption \ref{ass:ineq} allows for a geometric interpretation. Recall that by the Jordan-Brouwer separation theorem, a compact hypersurface $\M$ separates $\R^n$ into two connected sets, one interior set which is bounded, $\mathcal{K}$, and one exterior set which is unbounded $(\mathcal{K}\cup\mathcal{M})^\mathsf{c}$. The inequality in Assumption \ref{ass:ineq} implies that $\mathcal{K}$ is convex, \ie that $\M=\partial\mathcal{K}$ is the boundary of a convex set. To show this, we first need a lemma.

\begin{lemma}\label{le:dist}
Let $\mathcal{M}\subset\R^n$ be a nonsingular hypersurface given by $\M=\{\ve{y}\in\R^n\,|\,c(\ve{y})=0\}$, where $c$ is $\mathcal{C}^1$. Take any $\ve{z}\in\R^n$. The vector $\ve{v}$ of shortest length $\|\ve{v}\|$ such that $\ve{y}+\ve{v}=\ve{z}$ for some $\ve{y}\in\M$ is parallel to the normal of $\M$ given by $\nabla c(\ve{y})$.
\end{lemma}

\begin{proof}
The Lagrange conditions for optimality in the nonlinear optimization problem
\begin{align*}
\min_{y\in\R^n} \tfrac12\|\ve{z}-\ve{y}\|^2\textrm{ subject to } c(\ve{y})=0
\end{align*}	
are necessary since $\M$ is nonsingular (\ie all points on $\M$ are regular). Form the Lagrangian function $\mathcal{L}(\ve{y},\lambda)=\|\ve{z}-\ve{y}\|^2+\lambda c(\ve{y})$. The Lagrange conditions state that
\begin{align*}
	\ve{z}-\ve{y}+\lambda\nabla c(\ve{y})=\ve{0}
\end{align*}
from which it follows that $\ve{v}=\ve{z}-\ve{y}=-\lambda\nabla c(\ve{y})$.
\end{proof} 

\begin{theorem}\label{prop:convex}
Suppose Assumption \ref{ass:ineq} holds and that $\M$ is a closed manifold, then $\M$ is the boundary of a convex set.
\end{theorem}

\begin{proof}
Note that in order for Assumption \ref{ass:ineq} to hold, since $\langle\ve{n}(\ve{y}),\ve{n}(\ve{z})\rangle^2\leq1$, it is necessary that $\langle\ve{y}-\ve{z},\nabla{c}(\ve{y})\rangle$ and $\Delta c(\ve{y})-\langle\ve{n}(\ve{y}),\nabla^2c(\ve{y})\ve{n}(\ve{y})\rangle$ have the same sign. The latter expression only depends on $\ve{y}$ wherefore the sign of $\langle\ve{y}-\ve{z},\nabla{c}(\ve{y})\rangle$ cannot vary with $\ve{z}$, \ie either %Note that multiplying $c$ by a negative number  reverses the signs of these two quantities but does not affect $\M$. There is hence no loss of generality in assuming that either
\begin{align}\label{eq:convex1}
\langle\ve{y}-\ve{z},\nabla c(\ve{y})\rangle\geq0
\end{align}
or
\begin{align}\label{eq:convex2}
\langle\ve{y}-\ve{z},\nabla c(\ve{y})\rangle\leq0
\end{align}
holds for all $\ve{z}\in\M$ at any $\ve{y}\in\M$. 

Recall that we have chosen the sign of $c$ such that for all $\ve{y}\in\M$, $\nabla c(\ve{y})$  points towards the exterior of the two sets separated by $\M$. Let $\mathcal{K}$ denote the interior set. Following the negative normal $-\nabla c(\ve{y})$ on a line segment from $\ve{y}$ through the interior set $\mathcal{K}$, we find another point $\ve{z}\in\M$ (otherwise the interior set would be unbounded). Note that $\ve{y}-\ve{z}$ is aligned with the normal at $\ve{y}$. Hence $\langle\ve{y}-\ve{z},\nabla c(\ve{y})\rangle\geq0$ wherefore we can exclude the case of \eqref{eq:convex2}. 

By \eqref{eq:convex1}, for each $\ve{y}\in\M$, there is an affine hyperplane through $\ve{y}$ with normal $\nabla c(\ve{y})$. This hyperplane divides $\R^{n+1}$ into a closed set containing $\M$ and an open set which is disjoint from $\M$. Let $\mathcal{H}(\ve{y})$ denote the closed half-space which contains $\M$, \ie
\begin{align*}
\mathcal{H}(\ve{y})=\{\ve{w}\in\R^n\,|\,\langle\ve{y}-\ve{w},\nabla c(\ve{y})\rangle\geq0\}.
\end{align*}
Form
\begin{align*}
\mathcal{S}=\cap_{\ve{y}\in\M}\mathcal{H}(\ve{y}).
\end{align*}
Since $\mathcal{S}$ is an intersection of convex sets, it is convex. We will show that $\mathcal{K}=\mathcal{S}$.

Since $\M\subset\mathcal{H}(\ve{y})$ for all $\ve{y}\in\M$, it follows that $\mathcal{M}\subset\mathcal{S}$. Hence $\mathcal{K}\subset\mathcal{S}$. To show $\mathcal{S}\subset\mathcal{K}$, suppose by way of contradiction that there is a $\ve{s}\in\mathcal{S}\backslash\mathcal{K}$. There is a point $\ve{y}\in\M$ which minimizes the Euclidean distance to $\ve{s}$. By Lemma \ref{le:dist}, this point satisfies $\ve{s}=\ve{y}+\lambda\nabla c(\ve{y})$ for some $\lambda\in\R$. Because $\ve{s}\notin\mathcal{K}$ and $\nabla c(\ve{y})$ points away from $\mathcal{K}$ at $\ve{y}$, it must be the case that $\lambda>0$. Then 
\begin{align*}
\langle\ve{y}-\ve{s},\nabla c(\ve{y})\rangle=-\lambda\|\nabla c(\ve{y})\|^2<0.
\end{align*}
This implies that $\ve{s}\notin\mathcal{H}(\ve{y})$ and hence $\ve{s}\notin\mathcal{S}$, which contradicts the assumption that $\ve{s}\in\mathcal{S}\backslash\mathcal{K}$.\end{proof}

%\begin{remark}
%Recall that convexity of the feasible set is a desirable property of an optimization problem \cite{boyd2004convex}. Based on Proposition \ref{prop:convex}, we might think that if we relax the optimization problem \eqref{eq:min} by replacing $\M^N$ by its convex hull $\conv\M^N$ we would retain the same optimal points; \ie all optimal solutions to the relaxed problem lie on the boundary of $\conv\M^N$. This is false. Translate $\M^N$ so that the origin $\ve{0}$ lies inside $\M^N$. Take the line from any boundary point to $\ve{0}$. The objective function of \eqref{eq:min} decreases as we move towards the origin along the this line.
%\end{remark}

\subsection{Strongly convex functions}

\noindent Conversely, we could assume that $c$ is a convex function on all of $\R^n$. However, $c$ being convex does not imply that Assumption \ref{ass:convex} holds. A counter example is given by $c:\R^2\rightarrow\R:\ve{x}\mapsto \|\ve{x}\|^2-r^2$, which yields the Kuramoto model on $\mathcal{S}^1$. Consider the class of strongly convex functions. A strongly convex function $f$ with parameter $m$ satisfies
\begin{align*}
f(\ve{z})\geq f(\ve{y})+\langle\ve{z}-\ve{y},\nabla f(\ve{y})\rangle+\tfrac{m}{2}\|\ve{z}-\ve{y}\|^2
\end{align*}
at all points $\ve{y}$, $\ve{z}$ in its domain. That $c$ is strongly convex on $\M$ implies
\begin{align*}
\langle\ve{y}-\ve{z},\nabla c(\ve{y})\rangle\geq \tfrac{m}{2}\|\ve{z}-\ve{y}\|^2.
\end{align*}
Equivalently, any continuous function $f$ on a compact domain is strongly convex if  $m\ma{I}\preceq\nabla^2f(\ve{x})\preceq M\ma{I}$. 

\begin{assumption}\label{ass:convex}
Let $c$ be a strongly convex function, 
\begin{align*}
m\ma{I}\preceq\nabla^2 c(\ve{y})\preceq M\ma{I}. 
\end{align*}
Moreover, suppose that $c$ satisfies
\begin{align*}
	\tfrac{m((n+1)m-M)}{(LK)^2}\geq2,
\end{align*}
where $n=\dim\M$, $L$ is a global Lipschitz constant of the Gauss map $\ve{n}:\M\rightarrow\mathcal{S}^{n-1}$, \ie
\begin{align*}
	\|\ve{n}(\ve{y})-\ve{n}(\ve{z})\|\leq L\|\ve{y}-\ve{z}\|,
\end{align*}
for all $\ve{y},\ve{z}\in\M$, and $K=\max_{\ve{y}\in\M}\|\nabla c(\ve{y})\|$.
\end{assumption}

\begin{proposition}\label{le:convex}
Assumption \ref{ass:convex} implies Assumption \ref{ass:ineq}.
\end{proposition}	

\begin{proof} Consider the last term in the inequality of Assumption \ref{ass:ineq}. Strong convexity of $c$ implies that
\begin{align*}
\tfrac{\langle\ve{y}-\ve{z},\nabla c(\ve{y})\rangle(\Delta c(\ve{y})-\langle\ve{n}(\ve{y}),\nabla^2c(\ve{y})\ve{n}(\ve{y})\rangle)}{\|\nabla c(\ve{y})\|^2}\geq
\tfrac{m(nm-M)\|\ve{z}-\ve{y}\|^2}{2\|\nabla c(\ve{y})\|^2}.
\end{align*}
Since $\M$ is nonsingular by assumption, \ie $\nabla c(\ve{y})\neq\ve{0}$ for all $\ve{y}\in\M$, the Gauss map 
\begin{align*}
\ve{n}:\ve{y}\mapsto\tfrac{\nabla c(\ve{y})}{\|\nabla c(\ve{y})\|}
\end{align*}
is locally Lipschitz on $\M$.  Since $\M$ is a closed manifold there is a global Lipschitz constant $L$ of $\ve{n}$ over all points on $\M$. It follows that
\begin{align*}
\tfrac{m((n+1)m-M)\|\ve{z}-\ve{y}\|^2}{2\|\nabla c(\ve{y})\|^2}\geq
\tfrac{m((n+1)m-M)\|\ve{n}(\ve{z})-\ve{n}(\ve{y})\|^2}{2L^2K^2}
\end{align*}
where we also utilized the definition of $K$.

%Let $\vartheta$ denote the angle between $\ve{n}(\ve{y})$ and $\ve{n}(\ve{z})$. Then $\langle\ve[i]{n},\ve[j]{n}\rangle^2=\cos^2\vartheta$,e $\|\ve{n}(\ve{z})-\ve{n}(\ve{y})\|^2=2(1-\cos\vartheta)$, and
%

Let $\vartheta$ denote the angle between $\ve{n}(\ve{y})$ and $\ve{n}(\ve{z})$. For Assumption \eqref{ass:ineq} we find that
\begin{align*}
\langle\ve{n}(\ve{y}),\ve{n}(\ve{z})\rangle^2+\tfrac{\langle \ve{y}-\ve{z},\nabla c(\ve{y})\rangle(\Delta c(\ve{y})-\langle\ve{n}(\ve{y}),\nabla^2c(\ve{y})\ve{n}(\ve{y})\rangle)}{\|\nabla c(\ve{y})\|^2}&\geq\\
\cos^2\vartheta+\tfrac{m((n+1)m-M)\|\ve{z}-\ve{y}\|^2}{2K^2}&\geq\\
\cos^2\vartheta+\tfrac{m((n+1)m-M)\|\ve{n}(\ve{z})-\ve{n}(\ve{y})\|^2}{2L^2K^2}&=\\
\cos^2\vartheta+\tfrac{m((n+1)m-M)(1-\cos\vartheta)}{(LK)^2}&=\\
\cos^2\vartheta+\alpha(1-\cos\vartheta)&\geq1,
\end{align*}
where $\alpha=\tfrac{m((n+1)m-M)}{(LK)^2}$, if $\alpha$ is sufficiently large.

Denote $g(\vartheta,\alpha)=\cos^2\vartheta+\alpha(1-\cos\vartheta)$. We minimize this expression with respect to $\theta$ to find the range of $\alpha$ for which $g(\vartheta,\alpha)\geq1$ for all $\vartheta\in[0,\pi]$. Hence
\begin{align*}
\tfrac{\partial g(\vartheta,\alpha)}{\partial\vartheta}=-2\sin\vartheta\cos\vartheta+\alpha\sin\theta=0.
\end{align*}
Either $\sin\vartheta=0$ or $\cos\vartheta=\tfrac{\alpha}{2}$ for $\alpha\in[0,2]$. In the first case $\cos\theta\in\{-1,1\}$, which results in either $1+2\alpha\geq1$ or $1\geq1$. The condition on $\alpha$ is $\alpha\geq0$. In the second case 
\begin{align*}
g(\vartheta,\alpha)=\tfrac{\alpha^2}{4}+\alpha(1-\tfrac{\alpha}2)=\alpha-\tfrac{\alpha^2}{4}\geq1,
\end{align*}
which yields $\alpha\geq 2$. Hence we require $\tfrac{m((n+1)m-M)}{(LK)^2}\geq2$.\end{proof}

%Consider a family of hypersurfaces depending on the parameters $n$, $m$, $M$, $L$, and $K$. Suppose that all parameters except $n$ are constant. Examples of such manifolds are the ellipsoids given by 
%%
%\begin{align*}
%\M=
%\end{align*}
%%
%where $m$ and $M$ are constant. For any given $m$, $M$, the theorem implies that there is a $k\in\N$ such that $\mathcal{C}$ is \AGAS{} for all $n\geq k$.

%\section{Conclusions}

%\noindent This paper establishes a sufficient condition for almost global convergence of a consensus seeking multi-agent system on hypersurfaces. The hypersurfaces are assumed to be compact, analytic manifolds. The condition only holds if the hypersurface is the boundary of a convex set. We verify that the condition holds on ellipsoids.

%Simple connectedness for networks. While we can not explain why this would happen, we note that it seems to be the case.

\bibliographystyle{unsrt}
\bibliography{autosam}

\end{document}